\documentclass[a4paper]{article}

\usepackage{graphicx}
\usepackage{amsmath}
\usepackage{mathtools}
\usepackage{amsfonts,amssymb}
\usepackage{algorithmic}
\usepackage{algorithm}
\usepackage{array}
\usepackage{bm}
\usepackage[compress]{cite}
\usepackage[font=footnotesize]{caption}

\newtheorem{lemma}{\sc Lemma}[section]
\newtheorem{theorem}[lemma]{\sc Theorem}
\newtheorem{proposition}[lemma]{\sc Proposition}

\newtheorem{assumption}[lemma]{\bf Assumption}

\def\QED{~\rule[-1pt]{5pt}{5pt}\par\medskip}
\newenvironment{proof}{{\it Proof:\ }}{ \hfill \QED}

\floatname{algorithm}{Procedure}

\DeclareMathOperator*{\argmin}{arg\;min}
\newcommand{\h}{\hspace{.0001in}}
\renewcommand{\matrix}[2]{\left[\begin{array}{#1} #2 \end{array}\right] }

\begin{document}

\title{Distributed MPC Via Dual Decomposition and Alternating Direction Method of Multipliers\thanks{This work was supported in part by the Swedish Research Council, the Swedish Foundation for Strategic Research, and the Knut and Alice Wallenberg Foundation. }}

\author{Farhad Farokhi, Iman Shames, and Karl H. Johansson \thanks{ACCESS Linnaeus Center, School of Electrical Engineering, KTH Royal Institute of Technology, Emails:\{farokhi,imansh,kallej\}@ee.kth.se}}

\date{}

\maketitle

\abstract{A conventional way to handle model predictive control (MPC) problems distributedly is to solve them via dual decomposition and gradient ascent. However,  at each time-step, it might not be feasible to wait for the dual algorithm to converge. As a result, the algorithm might be needed to be terminated prematurely. One is then interested to see if the solution at the point of termination is close to the optimal solution and when one should terminate the algorithm if a certain distance to optimality is to be guaranteed. In this chapter, we look at this problem for distributed systems under general dynamical and performance couplings, then, we make a statement on validity of similar results where the problem is solved using alternating direction method of multipliers. }

\section{Introduction}
Model predictive control (MPC) can be used to control dynamical systems with input and output constraints  while ensuring the optimality of the performance of the system with respect to cost functions~\cite{rawlings2000tutorial,camponogara2002distributed,mattingley2011receding}. Typically, the way that the control input is calculated at each time-step is via applying the first control in a sequence obtained from solving an optimal control problem over a finite or infinite horizon. The optimal problem is  reformulated at each time step based on the available measurements at that time step. Traditionally, a full model of the system is required to solve the MPC problem and all the control inputs are calculated centrally. However, in large-scale interconnected systems, such as power systems~\cite{4682711,richalet1978model,negenborn2010intelligent}, water distribution systems~\cite{negenborn2010intelligent,zhang2008model}, transport systems~\cite{negenborn2008multi}, manufacturing systems~\cite{garcia1989model}, biological systems~\cite{hovorka2004nonlinear}, and irrigation systems~\cite{kearney2011model}, the assumption on knowing the whole model and calculating all the inputs centrally is often not realistic. Recently, much attention has been paid to solve MPC problems in a distributed way~\cite{motee2003optimal,jia2002min,dunbar2007distributed,venkat2005stability,Wakasa4739012,giselsson2010distributed,venkat2007distributed}.
The problem of distributed model predictive control using dual decomposition was considered in~\cite{Wakasa4739012}. However, in solving any optimization problem when using dual decomposition methods the convergence behaviors of dual iterations does not necessarily coincides to that of the primal formulation. Hence, the authors in~\cite{giselsson2010distributed} presented a distributed MPC algorithm using dual decomposition accompanied with a stopping criterion to guarantee a pre-specified level of performance. The authors only addressed linear coupled dynamics with separable cost functions.

In this chapter, specifically, we formulate the problem of achieving a control objective cooperatively by a network of dynamically coupled systems under constraints using MPC. We consider discrete-time nonlinear control systems. We are interested in casting the problem in a distributed way and we consider the case where the cost function associated with each system  is not necessarily decoupled from the rest. Additionally, we are not limiting our formulation to the case where the coupling in the cost function is the same as the coupling in the dynamics~\cite{Wakasa4739012,giselsson2010distributed}. We note that a natural method to solve such problems is to use dual-decomposition at each time-step and solve the problem iteratively. However, a problem that in implementing the dual solution iterations is that generally one cannot make any statement on how close the solution is to the optimum if the dual algorithm is terminated prematurely. That is, there is no termination guideline to ensure that the variables obtained from the dual algorithm are within an acceptable bound for the primal problem. In this chapter, we propose such termination guidelines that indicate how many iterations are needed to ensure a certain suboptimality guarantee, i.e.,~distance to optimality. We extend the results of~\cite{giselsson2010distributed} and present a more general frameworks, i.e., nonlinear interconnected dynamics and cost functions. A way to achieve better numerical properties for solving distributed MPC is to apply alternating direction method of multipliers (ADMM)~\cite{Boyd11}. Recently, optimal control synthesis and MPC via ADMM has gained some attention~\cite{bo2012,linfarjovACC12}. However, to the best of our knowledge, no attention has been paid to distributed MPC using ADMM. Hence, we show how to address distributed MPC via ADMM and how guarantees on termination errors can be obtained. We illustrate the applicability of our results on a formation of nonholonomic agents which employ distributed MPC to acquire a desired formation.

The outline of this chapter is as follows. In Section \ref{sec:prob} we formally define the problem of interest in this paper and particularly present the plant model and the performance criteria we consider. In Section \ref{sec:main} the suboptimality guarantee for the dually decomposed MPC is presented. Additionally, we make some comments on finding similar guarantee when the problem is solved via ADMM. We show the applicability of the results obtained here to the problem of controlling a formation of nonholonomic vehicles in Section \ref{sec:sim}. Finally, some concluding remarks are presented in Section~\ref{sec:conclusion}.

\subsection{Notation}
The sets of real and integer numbers are denoted by $\mathbb{R}$ and $\mathbb{Z}$, respectively. For any $n_1,n_2\in\mathbb{Z}\cup\{\pm \infty\}$, we define $\mathbb{Z}_{\geq n_1}^{\leq n_2}=\{n\in\mathbb{Z}\;|\;n_1\leq n\leq n_2\}.$ When $n_2=+\infty$, we use $\mathbb{Z}_{\geq n_1}$. For any $x\in\mathbb{R}$, we also define $\mathbb{R}_{\geq x}=\{y\in\mathbb{R}\;|\;y\geq x\}$. Other sets are denoted by calligraphic letters, such as $\mathcal{X}$ and $\mathcal{E}$.

Each (directed) graph is a pair of sets as $\mathcal{G}=(\mathcal{V},\mathcal{E})$, where $\mathcal{V}$ is the vertex set and $\mathcal{E}$ is the edge set. Each edge in the edge set $\mathcal{E}$ is an ordered pair of vertices, e.g., $(v_1,v_2)\in\mathcal{E}$.

\section{Problem Formulation} \label{sec:prob}
\subsection{Plant Model} \label{subsec:PM}
Let a directed graph $\mathcal{G}^P=(\{1,\dots,N\},\mathcal{E}^P)$ be given. Consider a discrete-time nonlinear control system composed of $N$ subsystems, where, for each $1\leq i\leq N$, subsystem~$i$ can be described in state-space form as
\begin{equation} \label{eqn:sys_dynamics}
\mathbf{x}_i[k+1]=\mathbf{f}_i(\mathbf{x}_i[k],\mathbf{v}_i[k];\mathbf{u}_i[k]),
\end{equation}
with state vector $\mathbf{x}_i[k]\in \mathcal{X}_i\subseteq \mathbb{R}^{n_i}$ and control input $\mathbf{u}_i[k]\in \mathcal{U}_i\subseteq \mathbb{R}^{m_i}$ for given integers $n_i,m_i\geq 1$. In addition, let $\mathbf{v}_i[k]= (\mathbf{x}_j[k])_{(j,i)\in\mathcal{E}^P} \in \mathbb{R}^{\sum_{(j,i)\in\mathcal{E}^P} n_j}$ denote the tuple of the state vector of all the subsystems that can influence subsystem~$i$ through its dynamics. For each $1\leq i\leq N$, mapping $\mathbf{f}_i:\mathcal{X}_i \times \prod_{(j,i)\in\mathcal{E}^P}\mathcal{X}_j \times \mathcal{U}_i \rightarrow \mathcal{X}_i$ determine the trajectory of subsystem~$i$ given the initial condition $x_i[0]\in\mathcal{X}_i$ and the inputs.

\subsection{Performance Criterion}
Let a directed graph $\mathcal{G}^C=(\{1,\dots,N\},\mathcal{E}^C)$ be given. For each time-instance $k\in\mathbb{Z}_{\geq 0}$, we introduce the running cost function
$$
J_k\left((\mathbf{x}_i[k])_{i=1}^N;(\mathbf{u}_i[k:+\infty])_{i=1}^N\right)=\sum_{t=k}^\infty\sum_{i=1}^N \boldsymbol{\ell}_i(\mathbf{x}_i[t],\mathbf{w}_i[t];\mathbf{u}_i[t]),
$$
where $\mathbf{w}_i[k]=(\mathbf{x}_j[k])_{(j,i)\in\mathcal{E}^C}\in\mathbb{R}^{\sum_{(j,i)\in\mathcal{E}^C} n_j}$ denotes the tuple of the state vector of all the subsystems that can influence subsystem~$i$ through its cost. Note that for the described dynamical system, given the control sequence $(\mathbf{u}_i[k:+\infty])_{i=1}^N$ and boundary condition $(\mathbf{x}_i[k])_{i=1}^N$, the trajectory of the system $(\mathbf{x}_i[k:+\infty])_{i=1}^N$ is uniquely determined by the described system dynamics in~(\ref{eqn:sys_dynamics}). Hence, we do not explicitly show the dependency of the cost function $J_k\left((\mathbf{x}_i[k])_{i=1}^N;(\mathbf{u}_i[k:+\infty])_{i=1}^N\right)$ to the trajectory $(\mathbf{x}_i[k+1:+\infty])_{i=1}^N$. We make the following standing assumption concerning the cost function which is crucial for proving stability of the origin for the closed-loop system with a MPC controller in feedback interconnection.

\begin{assumption} \label{asm:1} For each $1\leq i\leq N$, $\boldsymbol{\ell}_i:\mathcal{X}_i \times \prod_{(j,i)\in\mathcal{E}^C}\mathcal{X}_j \times \mathcal{U}_i \rightarrow \mathbb{R}_{\geq 0}$ is a mapping such that (a)~$\boldsymbol{\ell}_i(\mathbf{x}_i,\mathbf{w}_i;\mathbf{u}_i)$ is continuous in $\mathbf{x}_i$ for all $\mathbf{x}_i\in\mathcal{X}_i$ and (b)~$\boldsymbol{\ell}_i(\mathbf{x}_i,\mathbf{w}_i;\mathbf{u}_i)=0$ if and only if $\mathbf{x}_i=0$.
\end{assumption}

\subsection{MPC}\label{subsec:mpc}
In each time instance $k\in\mathbb{Z}_{\geq 0}$, the objective of the designer is to solve an infinite-horizon optimal control problem given by
\begin{equation} \label{eqn:MPC1}
\begin{split}
(\hat{\mathbf{u}}_i^*[k:+\infty])_{i=1}^N= \argmin_{ (\hat{\mathbf{u}}_i[k:+\infty])_{i=1}^N} \hspace{.1in}& J_k\left((\mathbf{x}_i[k])_{i=1}^N;(\hat{\mathbf{u}}_i[k:+\infty])_{i=1}^N\right),\\
\mbox{subject to } \hspace{.1in}& \hat{\mathbf{x}}_i[t+1]=\mathbf{f}_i(\hat{\mathbf{x}}_i[t],\hat{\mathbf{v}}_i[t];\hat{\mathbf{u}}_i[t]), 1\leq i\leq N, \forall\; t\in\mathbb{Z}_{\geq k}, \\ &\hat{\mathbf{x}}_i[k]=\mathbf{x}_i[k], \; 1\leq i\leq N, \\ & \hat{\mathbf{x}}_i[t]\in \mathcal{X}_i, \; 1\leq i\leq N, \forall\; t\in\mathbb{Z}_{\geq k}, \\ & \hat{\mathbf{u}}_i[t]\in \mathcal{U}_i, \; 1\leq i\leq N, \forall\; t\in\mathbb{Z}_{\geq k},
\end{split}
\end{equation}
where $(\hat{\mathbf{x}}_i[k:+\infty])_{i=1}^N$ is the state estimate initialized with the state measurement $\hat{\mathbf{x}}_i[k]=\mathbf{x}_i[k]$, for all $1\leq i\leq N$. Note that we use $\hat{\mathbf{x}}_i$ and $\hat{\mathbf{u}}_i$ to emphasize the fact that these variables are forecast variables and are predicted using the systems model. We relax the infinite-horizon optimal control problem in~(\ref{eqn:MPC1}) into a finite-horizon optimal control problem given by
\begin{equation} \label{eqn:MPC2}
\begin{split}
(\hat{\mathbf{u}}_i^*[k:k+T])_{i=1}^N= \argmin_{(\hat{\mathbf{u}}_i[k:k+T])_{i=1}^N} \hspace{.1in}& J_k^{(T)}\left((\mathbf{x}_i[k])_{i=1}^N;(\hat{\mathbf{u}}_i[k:k+T])_{i=1}^N\right),\\
\mbox{subject to } \hspace{.1in}& \hat{\mathbf{x}}_i[t+1]=\mathbf{f}_i(\hat{\mathbf{x}}_i[t],\hat{\mathbf{v}}_i[t];\hat{\mathbf{u}}_i[t]),  1\leq i\leq N, \forall\; t\in\mathbb{Z}_{\geq k}^{\leq k+T}, \\ &\hat{\mathbf{x}}_i[k]=\mathbf{x}_i[k], \; 1\leq i\leq N, \\ & \hat{\mathbf{x}}_i[t]\in \mathcal{X}_i, \; 1\leq i\leq N, \forall\; t\in\mathbb{Z}_{\geq k}^{\leq k+T}, \\ & \hat{\mathbf{u}}_i[t]\in \mathcal{U}_i, \; 1\leq i\leq N, \forall\; t\in\mathbb{Z}_{\geq k}^{\leq k+T},
\end{split}
\end{equation}
where
$$
J_k^{(T)}\left((\mathbf{x}_i[k])_{i=1}^N;(\hat{\mathbf{u}}_i[k:k+T])_{i=1}^N\right)=\sum_{t=k}^{k+T}\sum_{i=1}^N \boldsymbol{\ell}_i(\hat{\mathbf{x}}_i[t],\hat{\mathbf{w}}_i[t];\hat{\mathbf{u}}_i[t]),
$$
and $T\in\mathbb{Z}_{\geq 0}$ denotes the horizon of estimation and control. After solving this optimization problem, subcontroller~$i$ implements $\mathbf{u}_i[k]=\hat{\mathbf{u}}^*_i[k]$, for each $1\leq i\leq N$. Doing so, the overall cost of the system equals
$$
J_0 \left((\mathbf{x}_i[0])_{i=1}^N;(\mathbf{u}_i[0:+\infty])_{i=1}^N\right)=\sum_{t=0}^\infty\sum_{i=1}^N \boldsymbol{\ell}_i(\mathbf{x}_i[t],\mathbf{w}_i[t];\mathbf{u}_i[t]),
$$
where the control sequence $(\mathbf{u}_i[0:+\infty])_{i=1}^N$, as described earlier, is extracted step-by-step from the optimization problem in~(\ref{eqn:MPC2}). For the MPC problem to be well-posed, we make the following standing assumption:

\begin{assumption} \label{asm:2} The optimization problem
\begin{equation*}
\begin{split}
\argmin_{ (\hat{\mathbf{u}}_i[k:+\infty])_{i=1}^N} \hspace{.1in}& \sum_{t=k}^{k+T}\sum_{i=1}^N \boldsymbol{\ell}_i(\hat{\mathbf{x}}_i[t],\hat{\mathbf{w}}_i[t];\hat{\mathbf{u}}_i[t]),\\
\mathrm{subject\hspace{0.5em}to } \hspace{.1in}& \hat{\mathbf{x}}_i[t+1]=\mathbf{f}_i(\hat{\mathbf{x}}_i[t],\hat{\mathbf{v}}_i[t];\hat{\mathbf{u}}_i[t]), \; 1\leq i\leq N, \forall\; t\in\mathbb{Z}_{\geq k}^{\leq k+T}, \\ &\hat{\mathbf{x}}_i[k]=\mathbf{x}_i[k], \; 1\leq i\leq N, \\ & \hat{\mathbf{x}}_i[t]\in \mathcal{X}_i, \; 1\leq i\leq N, \forall\; t\in\mathbb{Z}_{\geq k}^{\leq k+T}, \\ & \hat{\mathbf{u}}_i[t]\in \mathcal{U}_i, \; 1\leq i\leq N, \forall\; t\in\mathbb{Z}_{\geq k}^{\leq k+T},
\end{split}
\end{equation*}
admits a unique global minimizer for all time horizon $T\in\mathbb{Z}_{\geq 0}\cup \{\infty\}$.
\end{assumption}

Assumption~\ref{asm:2} is evidently satisfied if, for each $1\leq i\leq N$,~(a)~mapping  $\boldsymbol{\ell}_i:\mathcal{X}_i \times \prod_{(j,i)\in\mathcal{E}^C}\mathcal{X}_j \times \mathcal{U}_i \rightarrow \mathbb{R}_{\geq 0}$ is quadratic, and~(b)~mapping $\mathbf{f}_i:\mathcal{X}_i \times \prod_{(j,i)\in\mathcal{E}^P}\mathcal{X}_j \times \mathcal{U}_i \rightarrow \mathcal{X}_i$ is linear~\cite{Anderson1971}. We can also consider strictly convex mappings $\boldsymbol{\ell}_i:\mathcal{X}_i \times \prod_{(j,i)\in\mathcal{E}^C}\mathcal{X}_j \times \mathcal{U}_i \rightarrow \mathbb{R}_{\geq 0}$ when working with finite-horizon cases~\cite{boyd2004convex}.

\section{Main Results}\label{sec:main}
Formulating a constrained optimization problem as a dual problem, in some cases, enables us to solve it in a decentralized manner across a network of agents. Typically, each iteration for solving the dual problem involves broadcasting and receiving variables for each agent. The variables that need to be communicated between the agents are the variables appearing in the cost function and the variables (Lagrange multipliers) used to enforce the constraints. In the rest of this section, we first cast the MPC problem in a dual decomposition framework and then introduce our result on guaranteeing the performance of the iterations to solve the decomposed problem distributedly.

\subsection{Dual Decomposition}
Let us, for each $1\leq i\leq N$, introduce slack variables $\bar{\mathbf{v}}_i[k]\in\mathbb{R}^{\sum_{(j,i)\in\mathcal{E}^P} n_j}$ and $\bar{\mathbf{w}}_i[k]\in\mathbb{R}^{\sum_{(j,i)\in\mathcal{E}^C} n_j}$. Doing so, we can rewrite the finite-horizon optimal control problem in~(\ref{eqn:MPC2}) as
\begin{equation*}
\begin{split}
(\hat{\mathbf{u}}_i^*[k:k+T])_{i=1}^N= \argmin_{(\hat{\mathbf{u}}_i[k:k+T])_{i=1}^N} \hspace{.1in}& \sum_{t=k}^{k+T}\sum_{i=1}^N \boldsymbol{\ell}_i(\hat{\mathbf{x}}_i[t],\bar{\mathbf{w}}_i[t];\hat{\mathbf{u}}_i[t]),\\
\mbox{subject to } \hspace{.1in}& \hat{\mathbf{x}}_i[t+1]=\mathbf{f}_i(\hat{\mathbf{x}}_i[t],\bar{\mathbf{v}}_i[t];\hat{\mathbf{u}}_i[t]), 1\leq i\leq N, \forall\; t\in\mathbb{Z}_{\geq k}^{\leq k+T}, \\ &\hat{\mathbf{x}}_i[k]=\mathbf{x}_i[k], \; 1\leq i\leq N, \\ & \hat{\mathbf{x}}_i[t]\in \mathcal{X}_i, \; 1\leq i\leq N, \forall\; t\in\mathbb{Z}_{\geq k}^{\leq k+T}, \\ & \hat{\mathbf{u}}_i[t]\in \mathcal{U}_i, \; 1\leq i\leq N, \forall\; t\in\mathbb{Z}_{\geq k}^{\leq k+T}, \\ & \bar{\mathbf{w}}_i[t]=\hat{\mathbf{w}}_i[t], 1\leq i\leq N, \forall\; t\in\mathbb{Z}_{\geq k}^{\leq k+T}, \\ & \bar{\mathbf{v}}_i[t]=\hat{\mathbf{v}}_i[t], 1\leq i\leq N, \forall\; t\in\mathbb{Z}_{\geq k}^{\leq k+T}.
\end{split}
\end{equation*}
We can incorporate the set of constraints $\bar{\mathbf{v}}_i[t]=\hat{\mathbf{v}}_i[t]$ and $\bar{\mathbf{w}}_i[t]=\hat{\mathbf{w}}_i[t]$ into the cost function as
\begin{equation} \label{eqn:Lagrange_cost}
\begin{split}
\max_{(\boldsymbol{\lambda}_i,\boldsymbol{\mu}_i)_{i=1}^N} \min_{ (\hat{\mathbf{u}}_i,\bar{\mathbf{v}}_i,\bar{\mathbf{w}}_i)_{i=1}^N } \sum_{t=k}^{k+T}\sum_{i=1}^N \bigg[ &\boldsymbol{\ell}_i(\hat{\mathbf{x}}_i[t],\bar{\mathbf{w}}_i[t];\hat{\mathbf{u}}_i[t]) \\&+\boldsymbol{\lambda}_i[t]^\top(\bar{\mathbf{v}}_i[t]-\hat{\mathbf{v}}_i[t])
+\boldsymbol{\mu}_i[t]^\top(\bar{\mathbf{w}}_i[t]-\hat{\mathbf{w}}_i[t])\bigg],
\end{split}
\end{equation}
where, for each $1\leq i\leq N$, variables $(\boldsymbol{\lambda}_i[k:k+T],\boldsymbol{\mu}_i[k:k+T])_{i=1}^N$ denote the Lagrange multipliers $\boldsymbol{\lambda}_i[t]=(\boldsymbol{\lambda}_{i,j}[t])_{(j,i) \in\mathcal{E}^P} \in\mathbb{R}^{\sum_{(j,i)\in\mathcal{E}^P} n_j}$
and $\boldsymbol{\mu}_i[t]=(\boldsymbol{\mu}_{i,j}[t])_{(j,i)\in\mathcal{E}^C} \in\mathbb{R}^{\sum_{(j,i)\in\mathcal{E}^C} n_j}$, for all $k\leq t\leq k+T$. Note that in~(\ref{eqn:Lagrange_cost}) we dropped the time index of the variables in the subscripts of the minimization and maximization operators to simplify the presentation. We can rearrange the cost function in~(\ref{eqn:Lagrange_cost}) as
\begin{equation} \label{eqn:Lagrange_cost:sperated}
\begin{split}
\max_{(\boldsymbol{\lambda}_i,\boldsymbol{\mu}_i)_{i=1}^N} \sum_{i=1}^{N} \min_{ \hat{\mathbf{u}}_i,\bar{\mathbf{v}}_i,\bar{\mathbf{w}}_i } \sum_{t=k}^{k+T}  \bigg[ \boldsymbol{\ell}_i(\hat{\mathbf{x}}_i[t],\bar{\mathbf{w}}_i[t]&;\hat{\mathbf{u}}_i[t])+ \boldsymbol{\lambda}_i[t]^\top \bar{\mathbf{v}}_i[t] +\boldsymbol{\mu}_i[t]^\top \bar{\mathbf{w}}_i[t]\\ &-\sum_{(i,j)\in \mathcal{E}^P} \boldsymbol{\lambda}_{j,i}[t]^\top \hat{\mathbf{x}}_i[t] -\sum_{(i,j)\in \mathcal{E}^C} \boldsymbol{\mu}_{j,i}[t]^\top \hat{\mathbf{x}}_i[t]\bigg].
\end{split}
\end{equation}
Using~(\ref{eqn:Lagrange_cost:sperated}), we can separate subsystem cost functions, which allows us to develop a distributed scheme for solving the finite-horizon MPC problem in~(\ref{eqn:MPC2}). This distributed scheme is presented in Procedure~\ref{alg:1}.

\begin{algorithm}[t!]
\caption{Distributed algorithm for solving the finite-horizon MPC problem~(\ref{eqn:MPC2})}
\begin{footnotesize}
\label{alg:1}
\begin{algorithmic}
\REQUIRE $\mathbf{x}_i[k]$, $1\leq i\leq N$
\ENSURE $\mathbf{u}_i[k]$, $1\leq i\leq N$
\STATE \hspace{-.115in}\textbf{Parameters:} Iteration numbers $\{S_k\}_{k=0}^\infty$ and gradient ascent step sizes $\{h_i^{(s)},g_i^{(s)}\}_{i,s=0}^\infty$
\FOR{$k=1,2,\dots$}
\STATE - Initialize Lagrange multipliers $(\boldsymbol{\lambda}_i^{(0)}[k:k+T],\boldsymbol{\mu}_i^{(0)}[k:k+T])_{i=1}^N$.
\FOR{$s=1,2,\dots,S_k$}
\FOR{$i=1,2,\dots,N$}
\STATE - Solve the optimization problem
\begin{equation*}
\begin{split}
(\hat{\mathbf{u}}_i^{(s)}[k:k+T],\bar{\mathbf{v}}_i\h^{(s)}[k:k+T], \bar{\mathbf{w}}_i&\h^{(s)}[k:k+T])\\= \argmin_{\hat{\mathbf{u}}_i,\bar{\mathbf{v}}_i,\bar{\mathbf{w}}_i}
\hspace{.1in} & L_i(\hat{\mathbf{u}}_i[k:k+T], \bar{\mathbf{v}}_i[k:k+T], \bar{\mathbf{w}}_i[k:k+T])\\
\mbox{subject to} \hspace{.1in} & \hat{\mathbf{x}}_i[t+1]=\mathbf{f}_i(\hat{\mathbf{x}}_i[t],\bar{\mathbf{v}}_i[t];\hat{\mathbf{u}}_i[t]),  \forall\; t\in\mathbb{Z}_{\geq k}^{\leq k+T}, \\ &\hat{\mathbf{x}}_i[k]=\mathbf{x}_i[k],  \\ & \hat{\mathbf{x}}_i[t]\in \mathcal{X}_i, \; \forall\; t\in\mathbb{Z}_{\geq k}^{\leq k+T}, \\ & \hat{\mathbf{u}}_i[t]\in \mathcal{U}_i, \;  \forall\; t\in\mathbb{Z}_{\geq k}^{\leq k+T}, \\ & \bar{\mathbf{v}}_i[t]\in \prod_{(j,i)\in\mathcal{E}^P}\mathcal{X}_j, \;  \forall\; t\in\mathbb{Z}_{\geq k}^{\leq k+T}, \\ &  \bar{\mathbf{w}}_i[t]\in\prod_{(j,i)\in\mathcal{E}^C}\mathcal{X}_j, \;  \forall\; t\in\mathbb{Z}_{\geq k}^{\leq k+T},
\end{split}
\end{equation*}
where
\begin{equation*}
\begin{split}
L_i(\hat{\mathbf{u}}_i[k:k+T], \bar{\mathbf{v}}_i[k:k+T], \bar{\mathbf{w}}_i[k:k+T])=\sum_{t=k}^{k+T} \bigg[& \boldsymbol{\ell}_i(\hat{\mathbf{x}}_i[t],\bar{\mathbf{w}}_i[t];\hat{\mathbf{u}}_i[t])+ \boldsymbol{\lambda}_i^{(s)}[t]^\top \bar{\mathbf{v}}_i[t] \\+\boldsymbol{\mu}_i^{(s)}[t]^\top \bar{\mathbf{w}}_i[t] &-\sum_{(i,j)\in \mathcal{E}^P} \boldsymbol{\lambda}_{j,i}^{(s)}[t]^\top \hat{\mathbf{x}}_i[t] -\sum_{(i,j)\in \mathcal{E}^C} \boldsymbol{\mu}_{j,i}^{(s)}[t]^\top \hat{\mathbf{x}}_i[t]\bigg].
\end{split}
\end{equation*}
\ENDFOR
\STATE - $\boldsymbol{\lambda}_i^{(s+1)}[t]=\boldsymbol{\lambda}_i^{(s)}[t]+h_i^{(s)} (\bar{\mathbf{v}}_i^{(s)}[t]-\hat{\mathbf{v}}_i^{(s)}[t]),\;1\leq i\leq N,\forall\; t\in\mathbb{Z}_{\geq k}^{\leq k+T}$.
\STATE - $\boldsymbol{\mu}_i^{(s+1)}[t]=\boldsymbol{\mu}_i^{(s)}[t]+g_i^{(s)} (\bar{\mathbf{w}}_i^{(s)}[t]-\hat{\mathbf{w}}_i^{(s)}[t]),\;1\leq i\leq N,\forall \; t\in\mathbb{Z}_{\geq k}^{\leq k+T}$.
\ENDFOR
\STATE - $\mathbf{u}_i[k]=\hat{\mathbf{u}}_i^{(S_k)}[k],\;1\leq i\leq N$.
\ENDFOR
\end{algorithmic}
\end{footnotesize}
\end{algorithm}

\subsection{From Infinite to Finite Horizon}
Let us introduce the notations
$$
V_k((\mathbf{x}_i[k])_{i=1}^N)=\min_{(\hat{\mathbf{u}}_i[k:+\infty])_{i=1}^N} J_k\left((\mathbf{x}_i[k])_{i=1}^N;(\hat{\mathbf{u}}_i[k:+\infty])_{i=1}^N\right),
$$
and
$$
V_k^{(T)}((\mathbf{x}_i[k])_{i=1}^N)=\min_{(\hat{\mathbf{u}}_i[k:k+T])_{i=1}^N} J_k^{(T)}\left((\mathbf{x}_i[k])_{i=1}^N;(\hat{\mathbf{u}}_i[k:k+T])_{i=1}^N\right),
$$
subject to the constraints introduced the infinite-horizon optimal control problem in~(\ref{eqn:MPC1}) and the finite-horizon optimal control problem in~(\ref{eqn:MPC2}), respectively.

\begin{theorem} \label{tho:1} Assume that there exist an a priori given constant $\alpha\in[0,1]$ and controllers $\boldsymbol{\phi}_i:\prod_{i=1}^N\mathcal{X}_i\rightarrow \mathcal{U}_i$, such that, for all $(\mathbf{x}_i[k])_{i=1}^N\in \prod_{i=1}^N\mathcal{X}_i$, we have
\begin{equation} \label{eqn:infinite_to_finite}
\begin{split}
V_k^{(T)}((\mathbf{x}_i[k])_{i=1}^N)\geq  V_{k+1}^{(T)}((\mathbf{x}_i[k+1])_{i=1}^N)+\alpha \sum_{i=1}^N \boldsymbol{\ell}_i(\mathbf{x}_i[k],\mathbf{w}_i[k];\boldsymbol{\phi}_i((\mathbf{x}_i[k])_{i=1}^N)),
\end{split}
\end{equation}
where $\mathbf{x}_i[k+1]=\mathbf{f}_i(\mathbf{x}_i[k],\mathbf{v}_i[k];\boldsymbol{\phi}_i((\mathbf{x}_i[k])_{i=1}^N))$, for each $1\leq i\leq N$. Then
$$
\alpha V_k((\mathbf{x}_i[k])_{i=1}^N) \leq V_k^{(T)}((\mathbf{x}_i[k])_{i=1}^N),
$$
for all $(\mathbf{x}_i[k])_{i=1}^N\in \prod_{i=1}^N\mathcal{X}_i$.
\end{theorem}

\begin{proof} The proof is a direct consequence of Proposition~2.2 in~\cite{Grune4639448}, when $\tilde{V}(\cdot)$ in~\cite{Grune4639448} is chosen equal to $V_k^{(T)}(\cdot)$ above.
\end{proof}

This theorem illustrates that by solving the finite-horizon optimal control problem in~(\ref{eqn:MPC2}), we get a sub-optimal solution, which is in a vicinity of the solution of the infinite-horizon optimal control problem in~(\ref{eqn:MPC1}) if $\alpha$ is chosen close to one. Hence, in the paper, we assume that the horizon $T$ is chosen such that it satisfies~(\ref{eqn:infinite_to_finite}). In that way, we do not lose much by abandoning the infinite-horizon optimal control problem for the finite-horizon one.

\subsection{Convergence} \label{subsec:dualdecomp}
Generically, in solving any optimization problem, if one resorts to use dual decomposition methods the convergence behaviors of dual iterations does not necessarily coincides to that of the primal formulation. In other words, if one terminates the dual iterations after $S_k$ steps and obtains the decision variables, one cannot make a statement on how close is the primal cost function evaluated at the obtained variable to its optimal value, i.e., the optimality gap cannot be determined. However, for model predictive control one can find a bound on such a distance. We aim to propose a way to calculate the optimality gap for general distributed MPC problems based on the results proposed by~\cite{giselsson2010distributed}.

Let us introduce the notation
\begin{equation*}
\begin{split}
V_k^{(T),(s)}((\mathbf{x}_i[k])_{i=1}^N)=\sum_{t=k}^{k+T}\sum_{i=1}^N \bigg[&\boldsymbol{\ell}_i(\hat{\mathbf{x}}_i[t],\bar{\mathbf{w}}_i\h^{(s)}[t]; \hat{\mathbf{u}}_i^{(s)}[t])\\&+\boldsymbol{\lambda}_i^{(s)}[t]^\top(\bar{\mathbf{v}}_i^{(s)}[t]- \hat{\mathbf{v}}_i[t]) +\boldsymbol{\mu}_i^{(s)}[t]^\top(\bar{\mathbf{w}}_i^{(s)}[t]-\hat{\mathbf{w}}_i[t])\bigg],
\end{split}
\end{equation*}
where $(\hat{\mathbf{u}}_i^{(s)}[k:k+T],\bar{\mathbf{v}}_i\h^{(s)}[k:k+T], \bar{\mathbf{w}}_i\h^{(s)}[k:k+T])_{i=1}^N$ is extracted from Procedure~1.

\begin{theorem} \label{tho:2} Let $\{\tilde{V}_k\}_{k=0}^\infty$ be a given family of mappings, such that, for each $k\in\mathbb{Z}_{\geq 0}$, $\tilde{V}_k:\prod_{i=1}^N\mathcal{X}_i\rightarrow \mathbb{R}_{\geq 0}$ satisfies
\begin{equation} \label{eqn:0:tho:2}
\tilde{V}_k((\mathbf{x}_i[k])_{i=1}^N) \geq V_k^{(T)}((\mathbf{x}_i[k])_{i=1}^N),
\end{equation}
for all $(\mathbf{x}_i[k])_{i=1}^N\in \prod_{i=1}^N\mathcal{X}_i$. In addition, let iteration number $S_k$ in Procedure~\ref{alg:1}, in each time-step $k\in\mathbb{Z}_{\geq 0}$, be given such that
\begin{equation} \label{eqn:1:tho:2}
V_k^{(T),(S_k)}((\mathbf{x}_i[k])_{i=1}^N)-\tilde{V}_{k+1}((\mathbf{x}_i[k+1])_{i=1}^N) \geq e[k]+\alpha \sum_{i=1}^N \boldsymbol{\ell}_i(\mathbf{x}_i[k],\mathbf{w}_i[k]; \hat{\mathbf{u}}_i^{(S_k)}[k])
\end{equation}
for a given constant $\alpha\in[0,1]$, where $\mathbf{x}_i[k+1]=\mathbf{f}_i(\mathbf{x}_i[k],\mathbf{v}_i[k];\hat{\mathbf{u}}_i^{(S_k)}[k])$, for each $1\leq i\leq N$. The sequence $\{e[k]\}_{k=0}^\infty$ is described by the difference equation
\begin{equation*}
\begin{split}
e[k]=e[k-1]+ \alpha \sum_{i=1}^N \boldsymbol{\ell}_i(\mathbf{x}_i[k-1]&,\mathbf{w}_i[k-1];\mathbf{u}_i[k-1])\\ &+\tilde{V}_k((\mathbf{x}_i[k])_{i=1}^N)
-\tilde{V}_{k-1}((\mathbf{x}_i[k-1])_{i=1}^N),
\end{split}
\end{equation*}
for all $k\in\mathbb{Z}_{\geq 2}$ and
\begin{equation*}
\begin{split}
e[1]=\alpha  \sum_{i=1}^N \boldsymbol{\ell}_i(\mathbf{x}_i[0],\mathbf{w}_i[0];\mathbf{u}_i[0])+\tilde{V}_1((\mathbf{x}_i[1])_{i=1}^N)
-V_0^{(T),(S_0)}((\mathbf{x}_i[0])_{i=1}^N),
\end{split}
\end{equation*}
and $e[0]=0$. Then
\begin{equation} \label{eqn:costguarantee}
\alpha J_0 \left((\mathbf{x}_i[0])_{i=1}^N;(\mathbf{u}_i[0:+\infty])_{i=1}^N\right) \leq V_0((\mathbf{x}_i[0])_{i=1}^N).
\end{equation}
for any initial condition $(\mathbf{x}_i[0])_{i=1}^N\in\prod_{i=1}^N\mathcal{X}_i$. In addition, if $V_0((\mathbf{x}_i[0])_{i=1}^N)<\infty$ for any initial condition $(\mathbf{x}_i[0])_{i=1}^N\in\prod_{i=1}^N\mathcal{X}_i$, then
$$
\lim_{k\rightarrow \infty} \mathbf{x}_i[k]=0,\; 1\leq i\leq N.
$$
\end{theorem}

\begin{proof} The proof of this theorem is a generalization of the proof of Theorem~3 in~\cite{giselsson2010distributed}. First, by induction, we can prove that
\begin{equation} \label{eqn:proof:1}
e[k]=\alpha \sum_{t=0}^{k-1}\sum_{i=1}^N \boldsymbol{\ell}_i(\mathbf{x}_i[t],\mathbf{w}_i[t];\mathbf{u}_i[t]) + \tilde{V}_k((\mathbf{x}_i[k])_{i=1}^N)
-V_0^{(T),(S_0)}((\mathbf{x}_i[0])_{i=1}^N).
\end{equation}
Substituting~(\ref{eqn:proof:1}) inside~(\ref{eqn:1:tho:2}), we get
\begin{equation} \label{eqn:proof:2}
\begin{split}
\alpha \sum_{t=0}^{k}\sum_{i=1}^N \boldsymbol{\ell}_i(\mathbf{x}_i[t],\mathbf{w}_i[t];\mathbf{u}_i[t]) &=e[k]+\alpha \sum_{i=1}^N \boldsymbol{\ell}_i(\mathbf{x}_i[k],\mathbf{w}_i[k];\hat{\mathbf{u}}_i^{(S_k)}[k])\\ &\hspace{.4in}-\tilde{V}_k((\mathbf{x}_i[k])_{i=1}^N)
+V_0^{(T),(S_0)}((\mathbf{x}_i[0])_{i=1}^N) \\ & \leq V_k^{(T),(S_k)}((\mathbf{x}_i[k])_{i=1}^N)-\tilde{V}_{k+1}((\mathbf{x}_i[k+1])_{i=1}^N)
\\ &\hspace{.4in}-\tilde{V}_k((\mathbf{x}_i[k])_{i=1}^N)
+V_0^{(T),(S_0)}((\mathbf{x}_i[0])_{i=1}^N).
\end{split}
\end{equation}
Considering condition~(\ref{eqn:0:tho:2}), we have
$$
\tilde{V}((\mathbf{x}_i[k])_{i=1}^N) \geq V_k^{(T)}((\mathbf{x}_i[k])_{i=1}^N) \geq V^{(T),(S_k)}((\mathbf{x}_i[k])_{i=1}^N),
$$
where the right-most inequality is a direct consequence of standard duality properties. Therefore, we can simplify~(\ref{eqn:proof:2}) into
\begin{equation*}
\begin{split}
\alpha \sum_{t=0}^{k}\sum_{i=1}^N \boldsymbol{\ell}_i(\mathbf{x}_i[t],\mathbf{w}_i[t];\mathbf{u}_i[t]) &\leq V_0^{(T),(S_0)}((\mathbf{x}_i[0])_{i=1}^N) -\tilde{V}_{k+1}((\mathbf{x}_i[k+1])_{i=1}^N) \\ & \leq V_0^{(T),(S_0)}((\mathbf{x}_i[0])_{i=1}^N) \\ & \leq V_0^{(T)}((\mathbf{x}_i[0])_{i=1}^N) \\ &\leq V_0((\mathbf{x}_i[0])_{i=1}^N),
\end{split}
\end{equation*}
where the last inequality is consequence of the observation that longer planning horizons results in larger cost functions (due to the lack of terminal cost or constraints)~\cite{giselsson2010distributed}. Hence, we get
\begin{equation*}
\begin{split}
\alpha J_0 \left((\mathbf{x}_i[0])_{i=1}^N;(\mathbf{u}_i[0:+\infty])_{i=1}^N\right) &= \lim_{k\rightarrow \infty} \alpha \sum_{t=0}^{k}\sum_{i=1}^N \boldsymbol{\ell}_i(\mathbf{x}_i[t],\mathbf{w}_i[t];\mathbf{u}_i[t]) \\ &\leq V_0((\mathbf{x}_i[0])_{i=1}^N).
\end{split}
\end{equation*}
The proof of the second part is done by a contradiction. Assume that there exists an index~$j$ such that $\mathbf{x}_j[k]$ does not converge to the origin as $k$ goes to infinity. Because of Assumption~\ref{asm:1}, it is easy to see that there exists a strictly increasing sequence of time-steps $\{k_\tau\}_{\tau=0}^\infty$ such that $\boldsymbol{\ell}_j(\mathbf{x}_j[k_\tau],\mathbf{w}_j[k_\tau]; \mathbf{u}_j[k_\tau])\geq \epsilon>0$ for all $\tau\in\mathbb{Z}_{\geq 0}$. Note that we can always construct such a sequence by enforcing the condition $\|\mathbf{x}_j[k_\tau]\|\geq \varepsilon'>0$ for all $\tau\in\mathbb{Z}_{\geq 0}$ (because otherwise, $\lim_{k\rightarrow \infty} \mathbf{x}_j[k]= 0$). Hence, we get
\begin{equation*}
\begin{split}
\alpha J_0 \left((\mathbf{x}_i[0])_{i=1}^N;(\mathbf{u}_i[0:+\infty])_{i=1}^N\right) &= \lim_{k\rightarrow \infty} \alpha \sum_{t=0}^{k}\sum_{i=1}^N \boldsymbol{\ell}_i(\mathbf{x}_i[t],\mathbf{w}_i[t];\mathbf{u}_i[t]) \\
&\geq \lim_{M\rightarrow \infty} \alpha \sum_{\tau=1}^{M} \boldsymbol{\ell}_j(\mathbf{x}_j[k_\tau],\mathbf{w}_j[k_\tau];\mathbf{u}_j[k_\tau])
\\ &\geq \lim_{M\rightarrow \infty} \sum_{\tau=0}^{M} \alpha \epsilon \\ &= \lim_{M\rightarrow \infty} \alpha \epsilon (M+1)=+\infty.
\end{split}
\end{equation*}
This contradicts the boundedness of $V_0((\mathbf{x}_i[0])_{i=1}^N)$.
\hfill$\square$
\end{proof}

Theorem~\ref{tho:2} shows that, provided that $\{S_k\}_{k=0}^\infty$ guarantees~(\ref{eqn:1:tho:2}), the cost of the sub-optimal controller extracted from Procedure~\ref{alg:1} is in a close vicinity of the global optimal controller, i.e., the cost of the sub-optimal controller is never worse than $1/\alpha$ times the cost of the global optimal controller. In addition, the closed-loop system is stable.

Now, we only need to present a mapping $\tilde{V}_k(\cdot)$ that satisfies condition~(\ref{eqn:0:tho:2}). We use the method presented in~\cite{giselsson2010distributed} for generating a reasonable upper bound. Let us introduce the one-step forward shift operator $\textbf{q}_T: (\prod_{i=1}^N\mathcal{U}_i)^{T+1} \rightarrow (\prod_{i=1}^N\mathcal{U}_i)^{T+1}$, so that for any control sequence $(\mathbf{u}_i[0:T])_{i=1}^N\in(\prod_{i=1}^N\mathcal{U}_i)^{T+1}$ we have
$$
\textbf{q}_T((\mathbf{u}_i[0:T])_{i=1}^N)=(\mathbf{u}'_i[0:T])_{i=1}^N,
$$
where
$$
\mathbf{u}'_i[t]=\left\{\begin{array}{ll} \mathbf{u}_i[t+1], & 0\leq t\leq T-1, \\ 0, & t=T,\end{array} \right.
$$
for all $1\leq i\leq N$. Now, for any time-step $k\in\mathbb{Z}_{\geq 1}$, we can define
$$
\tilde{V}_k((\mathbf{x}_i[k])_{i=1}^N)=J_k\left((\mathbf{x}_i[k])_{i=1}^N; \textbf{q}_T(\hat{\mathbf{u}}_i^{(S_{k-1})}[k-1:T+k-1])\right),
$$
where the control sequence $\hat{\mathbf{u}}_i^{(S_{k-1})}[k-1:T+k-1]$ denotes the control actions of step $k-1$ extracted from Procedure~\ref{alg:1}. For this described function, it is easy to check that
\begin{equation*}
\begin{split}
\tilde{V}_k((\mathbf{x}_i[k])_{i=1}^N)&=J_k^{(T)}\left((\mathbf{x}_i[k])_{i=1}^N; \textbf{q}_T(\hat{\mathbf{u}}_i^{(S_{k-1})}[k-1:T+k-1])\right)
\\& \geq V_k^{(T)}((\mathbf{x}_i[k])_{i=1}^N),
\end{split}
\end{equation*}
because the control sequence $\textbf{q}_T(\{\hat{\mathbf{u}}_i^{(S_{k-1})}[t]\}_{t=k-1}^{T+k-1})$ might not be optimal for time-step~$k$. Hence, we have proposed a suitable mapping for Theorem~\ref{tho:2}.

\subsection{ADMM Formulation}
A way to achieve better numerical properties for solving distributed MPC is to apply ADMM, which retains the decomposability of the dual formulation while ensuring better convergence properties in terms of speed and stability~\cite{Boyd11}. Recently, solving MPC via ADMM has gained some attention~\cite{bo2012}. In what comes next, we cast the problem introduced in this chapter in an ADMM framework and give a sub-optimality guarantee for this scenario.

We rewrite the MPC problem proposed in Section \ref{subsec:mpc}:
\begin{equation}\label{eq:min_admm_main}
\begin{split}
(\hat{\mathbf{u}}_i^*[k:k+T])_{i=1}^N= \argmin_{ (\hat{\mathbf{u}}_i[k:k+T])_{i=1}^N} \hspace{.1in}& \sum_{t=k}^{k+T}\sum_{i=1}^N \boldsymbol{\ell}_i(\mathbf{y}_i[t]),\\
\mbox{subject to } \hspace{.1in}& \mathbf{y}_i[t] \in\mathcal{C}_i[t],\;1\leq i\leq N, \;\forall\; t\in\mathbb{Z}_{\geq k}^{\leq k+T}
\end{split}
\end{equation}
where, for each $k\leq t\leq k+T$, $\mathbf{y}_i[t]= [\hat{\mathbf{x}}_i[t]^\top,\bar{\mathbf{w}}_i[t]^\top,\hat{\mathbf{u}}_i[t]^\top, \bar{\mathbf{v}}_i[t]]^\top$, and
\begin{equation}
\begin{split}
\mathcal{C}_i[t]=\bigg\{[\hat{\mathbf{x}}_i[t]^\top,&\bar{\mathbf{w}}_i[t]^\top,\hat{\mathbf{u}}_i[t]^\top, \bar{\mathbf{v}}_i[t]]^\top \;\big|\; \hat{\mathbf{x}}_i[s+1] =\mathbf{f}_i(\hat{\mathbf{x}}_i[s], \bar{\mathbf{v}}_i[s];\hat{\mathbf{u}}_i[s]),\; \hat{\mathbf{x}}_i[s]\in \mathcal{X}_i,\\ &\hat{\mathbf{u}}_i[s]\in \mathcal{U}_i, \;\bar{\mathbf{w}}_i[s]=\hat{\mathbf{w}}_i[s], \; \bar{\mathbf{v}}_i[s]=\hat{\mathbf{v}}_i[s],\;\forall\; s\in\mathbb{Z}_{\geq k}^{\leq t},\; \hat{\mathbf{x}}_i[k]=\mathbf{x}_i[k]\bigg\}.
  \end{split}
\end{equation}
Provided that $\mathcal{C}_i[t]$ is convex for all $t\in\mathbb{Z}_{\geq k}^{\leq k+T}
$, then \eqref{eq:min_admm_main} is equivalent to
\begin{equation*}
\begin{split}
(\hat{\mathbf{u}}_i^*[k:k+T])_{i=1}^N= \argmin_{ (\hat{\mathbf{u}}_i[k:k+T])_{i=1}^N} \hspace{.1in}& \sum_{t=k}^{k+T}\sum_{i=1}^N \big[ \boldsymbol{\ell}_i(\mathbf{y}_i[t])+\boldsymbol{\psi}_i(\boldsymbol{\zeta}_i[t])\big],\\
\mbox{subject to } \hspace{.1in}& \mathbf{y}_i[t] =\boldsymbol{\zeta}_i[t], \;1\leq i\leq N, \;\forall\; t\in\mathbb{Z}_{\geq k}^{\leq k+T}
\end{split}
\end{equation*}
where  $\boldsymbol{\psi}_{\mathcal{C}_i[t]}(\cdot)$ is an indicator for $\mathcal{C}_i[t]$, viz. $\boldsymbol{\psi}_{\mathcal{C}_i[t]}(z)=0$ if $z\in\mathcal{C}_i[t]$ and $\boldsymbol{\psi}_{\mathcal{C}_i[t]}(z)=+\infty$ otherwise. The augmented Lagrangian for this problem is
\begin{equation}
\begin{split}
L((\mathbf{y}_i[k:k+T]&)_{i=1}^N,(\boldsymbol{\zeta}_i[k:k+T])_{i=1}^N,(\boldsymbol{\gamma}_i[k:k+T])_{i=1}^N)\\&=
\sum_{t=k}^{k+T}\sum_{i=1}^N \bigg[ \boldsymbol{\ell}_i(\mathbf{y}_i[t])+\boldsymbol{\psi}_i(\boldsymbol{\zeta}_i[t]) +\dfrac{\rho}{2}\left\|\mathbf{y}_i[t] -\boldsymbol{\zeta}_i[t]-\boldsymbol{\gamma}_i[t]\right\|^2\bigg],
\end{split}
\end{equation}
where $\boldsymbol{\gamma}_i[t]$, $1\leq i\leq N$, are the scaled dual variables.
We outline a distributed procedure that solves the problem in Procedure~\ref{alg:2}.

\begin{algorithm}[t!]
\caption{Distributed algorithm for solving the MPC problem~(\ref{eqn:MPC2}) via ADMM}
\begin{footnotesize}
\label{alg:2}
\begin{algorithmic}
\REQUIRE $\mathbf{x}_i[k]$, $1\leq i\leq N$
\ENSURE $\mathbf{u}_i[k]$, $1\leq i\leq N$
\FOR{$k=1,2,\dots$}
\STATE - Initialize scaled dual variables $(\boldsymbol{\gamma}_i^{(0)}[t])_{i=1}^N$.
\FOR{$s=1,2,\dots,S_k$}
\FOR{$i=1,2,\dots,N$}
\STATE - Solve the optimization problem
\begin{equation*}
\begin{split}
\mathbf{y}_i^{(s)}[k:k+T]= \argmin_{ \mathbf{y}_i[k:k+T]} & \sum_{t=k}^{k+T} \bigg[ \boldsymbol{\ell}_i(\mathbf{y}_i[t])+\dfrac{\rho}{2}\|\mathbf{y}_i[t] -\boldsymbol{\zeta}^{(s)}_i[t]-\boldsymbol{\gamma}^{(s)}_i[t]\|^2\bigg].\end{split}
\end{equation*}
\ENDFOR
\STATE - $\boldsymbol{\zeta}_i^{(s+1)}[t]=\Pi_{\mathcal{C}_i[t]}(\mathbf{y}_i[t]+\boldsymbol{\gamma}_i^{(s)}[t]),\;1\leq i\leq N,\forall\; t\in\mathbb{Z}_{\geq k}^{\leq k+T}$. \COMMENT{$\Pi_{\mathcal{C}_i[t]}(\cdot)$ is a projection onto $\mathcal{C}_i[t]$}
\STATE - $\boldsymbol{\gamma}_i^{(s+1)}[t] := \boldsymbol{\gamma}_i^{(s)}[t] + (\mathbf{y}_i[t] -\boldsymbol{\zeta}_i^{(s+1)}[t]),\;1\leq i\leq N,\forall \; t\in\mathbb{Z}_{\geq k}^{\leq k+T}$.
\ENDFOR
\STATE - $\mathbf{u}_i[k]=\hat{\mathbf{u}}_i^{(S_k)}[k],\;1\leq i\leq N$.
\ENDFOR
\end{algorithmic}
\end{footnotesize}
\end{algorithm}

Now, we reintroduce the following function for solving the MPC problem via ADMM:
\begin{equation*}
\begin{split}
V^{(T),(s)}((\mathbf{x}_i[k])_{i=1}^N)=\sum_{t=k}^{k+T}\sum_{i=1}^N \bigg[\boldsymbol{\ell}_i&(\mathbf{y}_i^{(s)}[t])+\dfrac{\rho}{2}\left\|\mathbf{y}_i^{(s)}[t] -\boldsymbol{\zeta}_i^{(s)}[t]-\boldsymbol{\gamma}_i^{(s)}[t]\right\|^2\bigg].
\end{split}
\end{equation*}

For the case that $\|\mathbf{y}_i^{(s)}[t] -\boldsymbol{\zeta}_i^{(s)}[t]-\boldsymbol{\gamma}_i^{(s)}[t]\|^2$ is less than a given threshold $\varepsilon \ll 1$, one might be able to follow the same line of reasoning as in Theorem~\ref{tho:2}. The next proposition can be seen as a simple example of such results.

\begin{proposition} \label{tho:3} Assume that, in Procedure~\ref{alg:2}, for any time-step $k\in\mathbb{Z}_{\geq 0}$, we have
\begin{equation} \label{eqn:condition:ADMM:convergence}
\lim_{S_k\rightarrow +\infty}V_k^{(T),(S_k)}((\mathbf{x}_i[k])_{i=1}^N)= V_k^{(T)}((\mathbf{x}_i[k])_{i=1}^N).
\end{equation}
Let $\{\tilde{V}_k\}_{k=0}^\infty$ be a given family of mappings $\tilde{V}_k:\prod_{i=1}^N\mathcal{X}_i\rightarrow \mathbb{R}_{\geq 0}$ for $k\in\mathbb{Z}_{\geq 0}$. In addition, let iteration number $S_k$ in Procedure~\ref{alg:2}, in each time-step $k\in\mathbb{Z}_{\geq 0}$, be given such that
\begin{equation} \label{eqn:1:tho:3:1}
\tilde{V}_{k}((\mathbf{x}_i[k])_{i=1}^N) \geq V_k^{(T),(S_k)}((\mathbf{x}_i[k])_{i=1}^N),
\end{equation}
and
\begin{equation} \label{eqn:1:tho:3:2}
V_k^{(T),(S_k)}((\mathbf{x}_i[k])_{i=1}^N)-\tilde{V}_{k+1}((\mathbf{x}_i[k+1])_{i=1}^N) \geq e[k]+\alpha \sum_{i=1}^N \boldsymbol{\ell}_i(\mathbf{x}_i[k],\mathbf{w}_i[k]; \hat{\mathbf{u}}_i^{(S_k)}[k]),
\end{equation}
for a given constant $\alpha\in[0,1]$, where $\mathbf{x}_i[k+1]=\mathbf{f}_i(\mathbf{x}_i[k], \mathbf{v}_i[k];\hat{\mathbf{u}}_i^{(S_k)}[k])$, for each $1\leq i\leq N$. The sequence $\{e[k]\}_{k=0}^\infty$ is described by the difference equation
\begin{equation*}
\begin{split}
e[k]=e[k-1]+ \alpha \sum_{i=1}^N \boldsymbol{\ell}_i(\mathbf{x}_i[k-1]&,\mathbf{w}_i[k-1];\mathbf{u}_i[k-1])\\ &+\tilde{V}_k((\mathbf{x}_i[k])_{i=1}^N)
-\tilde{V}_{k-1}((\mathbf{x}_i[k-1])_{i=1}^N),
\end{split}
\end{equation*}
for all $k\in\mathbb{Z}_{\geq 2}$ and
\begin{equation*}
\begin{split}
e[1]=\alpha \sum_{i=1}^N \boldsymbol{\ell}_i(\mathbf{x}_i[0],\mathbf{w}_i[0];\mathbf{u}_i[0])+\tilde{V}_1((\mathbf{x}_i[1])_{i=1}^N)
-V_0^{(T),(S_0)}((\mathbf{x}_i[0])_{i=1}^N),
\end{split}
\end{equation*}
and $e[0]=0$. Then, for any initial condition $(\mathbf{x}_i[0])_{i=1}^N\in\prod_{i=1}^N\mathcal{X}_i$ and any given constant $\varepsilon>0$, there exists $S_0\in\mathbb{Z}_{\geq 1}$ such that
\begin{equation} \label{eqn:costguarantee3}
\alpha J_0 \left((\mathbf{x}_i[0])_{i=1}^N;(\mathbf{u}_i[0:+\infty])_{i=1}^N\right) \leq V_0((\mathbf{x}_i[0])_{i=1}^N)+\varepsilon.
\end{equation}
\end{proposition}

\begin{proof} Similar to Theorem~\ref{tho:2}, we can prove that
\begin{equation*}
\begin{split}
\alpha \sum_{t=0}^{k}\sum_{i=1}^N \boldsymbol{\ell}_i(\mathbf{x}_i[t],\mathbf{w}_i[t];\mathbf{u}_i[t]) &\leq V_0^{(T),(S_0)}((\mathbf{x}_i[0])_{i=1}^N).
\end{split}
\end{equation*}
Now, noticing that $\lim_{S_0\rightarrow \infty}V_0^{(T),(S_0)}((\mathbf{x}_i[0])_{i=1}^N) =V_0^{(T)}((\mathbf{x}_i[0])_{i=1}^N)$, the rest of the proof follows.
\hfill$\square$
\end{proof}

Note that in general it is not easy to verify condition~(\ref{eqn:condition:ADMM:convergence}). However, we believe it is interesting to be able to provide guarantees for the closed-loop performance of distributed MPC in the case where condition~(\ref{eqn:condition:ADMM:convergence}) holds.

It is important to show that, for an appropriate family of mappings, we can satisfy condition~(\ref{eqn:1:tho:3:1}) with finite number of iterations. Let us use a similar approach as the one we used in Subsection~\ref{subsec:dualdecomp} for calculating a family of mappings~$\{\tilde{V}_k\}_{k=1}^\infty$. Doing so, for each time-step $k\in\mathbb{Z}_{\geq 1}$, we define
$$
\tilde{V}_k((\mathbf{x}_i[k])_{i=1}^N)=J_k\left((\mathbf{x}_i[k])_{i=1}^N; \textbf{q}_T(\hat{\mathbf{u}}_i^{(S_{k-1})}[k-1:T+k-1])\right)+\vartheta_k,
$$
where the sequence $\{\vartheta_k\}_{k=0}^\infty$ is fixed and such that $\vartheta_k>0$ for all $k\in\mathbb{Z}_{\geq 0}$. Hence, we get $\tilde{V}_k((\mathbf{x}_i[k])_{i=1}^N)\geq V^{(T)}((\mathbf{x}_i[k])_{i=1}^N)$. Considering that  $\lim_{s\rightarrow +\infty} V_k^{(T),(s)}((\mathbf{x}_i[k])_{i=1}^N) =V^{(T)} \linebreak[4] ((\mathbf{x}_i[k])_{i=1}^N),$ there always exists $S'_k\in\mathbb{Z}_{\geq 0}$ such that
\begin{equation} \label{eqn:new:condition}
\tilde{V}_{k}((\mathbf{x}_i[k])_{i=1}^N)-V^{(T)}((\mathbf{x}_i[k])_{i=1}^N) \geq V_k^{(T),(s)}((\mathbf{x}_i[k])_{i=1}^N)-V^{(T)}((\mathbf{x}_i[k])_{i=1}^N),
\end{equation}
for all $s\geq S'_k$. Note that the inequality in~(\ref{eqn:new:condition}) is equivalent to the inequality in~(\ref{eqn:1:tho:3:1}). The sequence $\{\vartheta_k\}_{k=0}^\infty$ is a design parameter. If $\vartheta_k$ are chosen relatively small, it is difficult to satisfy~(\ref{eqn:1:tho:3:1}) with a few iteration, whereas if $\vartheta_k$ are chosen relatively large, it is difficult to satisfy~(\ref{eqn:1:tho:3:2}) because $e[k]$ would become large.

\begin{figure}[t]
\centering
\includegraphics[width=.6\linewidth]{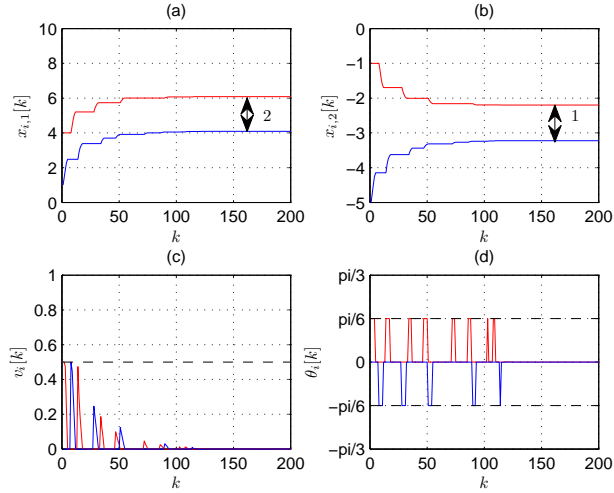}
\caption{\label{figure1} Trajectory and control signal of two vehicles when using Procedure~\ref{alg:1} and termination law described in Theorem~\ref{tho:2}. }
\end{figure}
\begin{figure}[t]
\centering
\includegraphics[width=.6\linewidth]{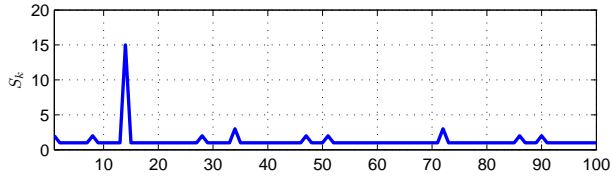}
\caption{\label{figureiterationnumber} Iteration numbers $S_k$ versus time-step~$k$ for the simulation results in Figure~\ref{figure1}. }
\end{figure}

\section{Simulations}\label{sec:sim}
In this section, we portray the applicability of the algorithm to a formation acquisition problem. We assume that the nonholonmic vehicle~$i$, for each $i=1,\dots,N$, can be described in state-space representation as
$$
\mathbf{x}_i[k+1]=\mathbf{x}_i[k]+\matrix{c}{\mathbf{v}_i[k]\cos(\theta_i[k]) \\ \mathbf{v}_i[k]\sin(\theta_i[k])},
$$
where $\mathbf{x}_i[k]=\matrix{cc}{x_{i,1}[k] & x_{i,2}[k]}^\top\in\mathbb{R}^2$ is the position of the vehicle, $\mathbf{v}_i[k]\in\mathbb{R}$ is its velocity, and $\theta_i[k]\in\mathbb{R}$ is its steering-wheel angle. Because of the vehicles' mechanical constraints, i.e., bounded speed and steering angle, we assume that the control inputs should always remain bounded as $0\leq \mathbf{v}_i[k]\leq 0.5$ and $|\theta_i[k]|\leq \pi/6$ for all $k\in\mathbb{Z}_{\geq 0}$.
We define each vehicle control input as $\mathbf{u}_i[k]=\matrix{cc}{\mathbf{v}_i[k] & \theta_i[k] }^\top $. Let us start with two vehicles. At each time-step $k\in\mathbb{Z}_{\geq 0}$, these vehicles are interested in minimizing the cost function
\begin{equation*}
\begin{split}
J_k^{(T)}((\mathbf{x}_i[k])_{i=1}^2;(\mathbf{u}_i[k:k+T])_{i=1}^2)=\sum_{t=k}^{k+T} \bigg[ 2\|\mathbf{x}_1[t]-\mathbf{x}_2[t]-\mathbf{d}_{12}\|_2^2+10(v_1^2[t]+v_2^2[t])\bigg],
\end{split}
\end{equation*}
where $\mathbf{d}_{12}=\matrix{cc}{2 & 1}^\top$. Let us fix the starting points of the vehicles as
$$
\mathbf{x}_1[0]=\matrix{c}{+4.0 \\ -1.0},\hspace{.3in}\mathbf{x}_2[0]=\matrix{c}{+1.0 \\ -5.0}.
$$
We also fix the planning horizon $T=5$. We use Procedure~\ref{alg:1} to calculate the sequence of sub-optimal control signals when the termination law (i.e., iteration number $S_k$) is given by Theorem~\ref{tho:2}. Figure~\ref{figure1} illustrates the trajectory and the control signals of both vehicles with finite-horizon planning when the sub-optimality parameter is fixed at $\alpha=0.5$. To be precise, Figures~\ref{figure1}(a,b) portray different coordinates of the vehicle position while Figures~\ref{figure1}(c,d) illustrate the velocities and steering-wheel angels of the vehicles, respectively. The red color denotes the first vehicle and the blue color denotes the second one. The portrayed simulation is done for 200 time-steps. It is interesting to note that over the first 100 time-steps, in average 1.25~iterations per time-step were used in Procedure~\ref{alg:1} to extract the sub-optimal control signal (see Figure~\ref{figureiterationnumber}). Figure~\ref{figure2} illustrates the trajectory and the control signals of both vehicles with finite-horizon planning using a centralized optimization algorithm as a reference. We also check the influence of $\alpha$. To do so, we introduce some notations. For each time-step~$k$, we define $(\mathbf{u}_i^{\mathrm{Primal}}[k])_{i=1}^2
=(\hat{\mathbf{u}}_i[k])_{i=1}^2$ where
\begin{equation*}
\begin{split}
(\hat{\mathbf{u}}_i[k:k+T])_{i=1}^2=\argmin_{(\hat{\mathbf{u}}_i[k:k+T]\in\mathcal{U}_i)_{i=1}^2} J_k^{(T)}((\mathbf{x}_i[k])_{i=1}^2;(\hat{\mathbf{u}}_i[k:k+T])_{i=1}^2).
\end{split}
\end{equation*}
Similarly, we define $(\mathbf{u}_i^{\mathrm{Dual}}[k])_{i=1}^2= (\hat{\mathbf{u}}_i^{(S_k)}[k])_{i=1}^2$ where, for each vehicle, $\hat{\mathbf{u}}_i^{(S_k)}[k]$ is calculated using Procedure~\ref{alg:1} when the dual decomposition iteration numbers $\{S_k\}_{k=0}^\infty$ is extracted from Theorem~\ref{tho:2}. Now, we define the ratio
$$
\rho=\frac{J_0^{(H)}\left((\mathbf{x}_i[0])_{i=1}^2;(\mathbf{u}_i^{\mathrm{Dual}}[0:H])_{i=1}^2\right)} {J_0^{(H)}\left((\mathbf{x}_i[0])_{i=1}^2;(\mathbf{u}_1^{\mathrm{Primal}}[0:H])_{i=1}^2\right)},
$$
where $H$ is the simulation horizon. Table~\ref{table1} shows $\rho$ as a function of $\alpha$ for $H=1000$. Based on this table, we can numerically verify the claim of Theorem~\ref{tho:2} that using Procedure~\ref{alg:1} when the dual decomposition iteration numbers is extracted from~(\ref{eqn:1:tho:2}) provides a suboptimality ratio~$\rho$ that is inversely proportional to the design parameter~$\alpha$.

\begin{figure}[t]
\centering
\includegraphics[width=.6\linewidth]{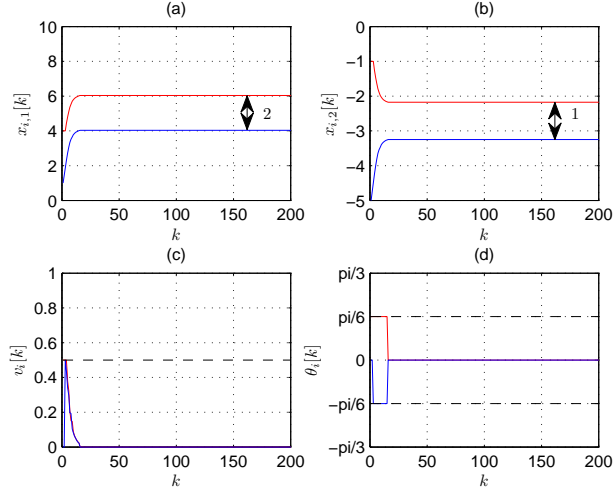}
\caption{\label{figure2} Trajectory and control signal of two vehicles when using a centralized optimization algorithm. }
\end{figure}

\begin{table}[t]
\centering
\caption{\label{table1} Sub-optimality ratio as function of $\alpha$. }
\footnotesize
\setlength{\extrarowheight}{3pt}
\begin{tabular}{|c|c|c|c|c|}
\hline
$\alpha$ & $0.1$ & $0.3$ & $0.5$ & $0.7$ \\
\hline\hline
\hspace{.1in} $\rho$ \hspace{.1in}
& \hspace{.1in}9.7875\hspace{.1in} & \hspace{.1in}3.2725\hspace{.1in} & \hspace{.1in}1.9684\hspace{.1in} & \hspace{.1in}1.4198\hspace{.1in} \\
\hline
\end{tabular}
\end{table}

\begin{figure}[t]
\centering
\includegraphics[width=.6\linewidth]{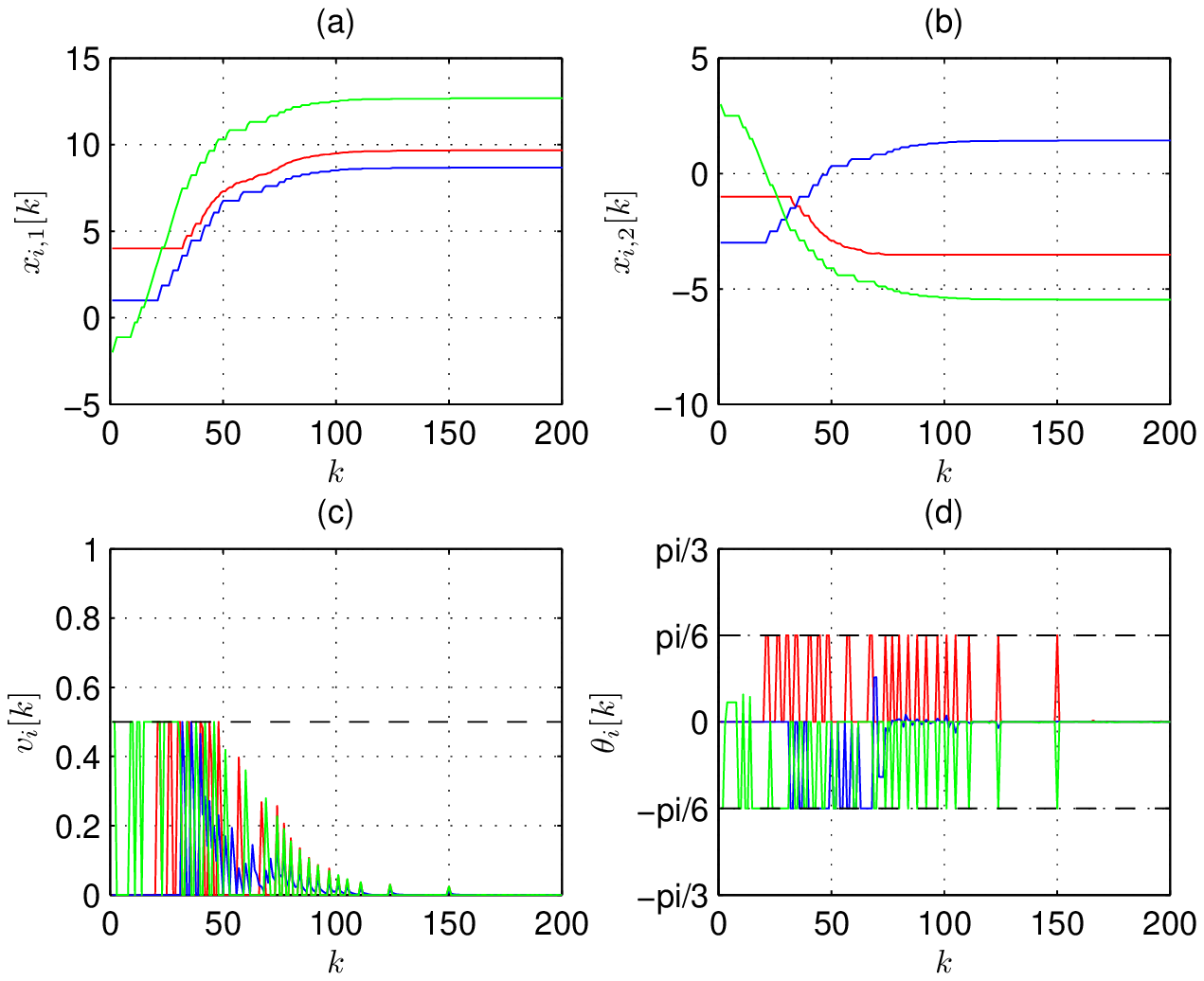}
\caption{\label{figure3} Trajectory and control signal of three vehicles when using Procedure~\ref{alg:1} and termination law described in Theorem~\ref{tho:2}. }
\end{figure}

\begin{figure}[t]
\centering
\includegraphics[width=.6\linewidth]{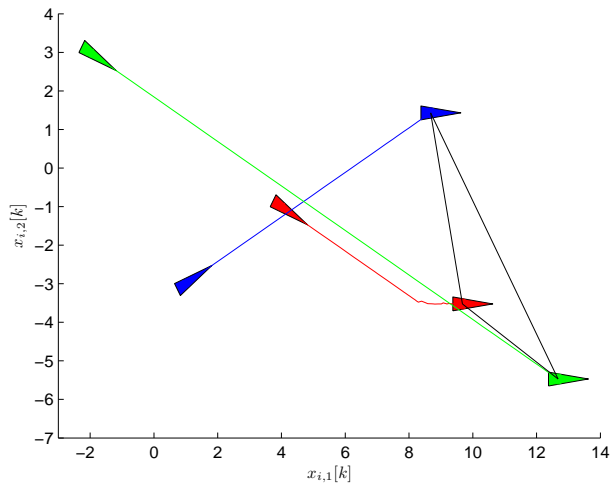}
\caption{\label{figuretrajectory} Trajectory of the vehicles in the 2-D plane. }
\end{figure}

These simulations can be readily extended to more vehicles. Figure~\ref{figure3} illustrates the trajectory and the control signals of three vehicles when trying to minimize the cost function
\begin{equation*}
\begin{split}
J_k^{(T)}((\mathbf{x}_i[k])_{i=1}^3;(\mathbf{u}_i[k:k+T])_{i=1}^3)=&\sum_{t=k}^{k+T} \bigg[ 2\|\mathbf{x}_1[t]-\mathbf{x}_2[t]-\mathbf{d}_{12}\|_2^2+
2\|\mathbf{x}_1[t]-\mathbf{x}_3[t]-\mathbf{d}_{13}\|_2^2\\&\hspace{.3in}+
2\|\mathbf{x}_2[t]-\mathbf{x}_3[t]-\mathbf{d}_{23}\|_2^2+
10(v_1^2[t]+v_2^2[t]+v_2^2[t])\bigg],
\end{split}
\end{equation*}
with
$$
\mathbf{d}_{12}=\matrix{c}{+1 \\ -5}, \hspace{.2in} \mathbf{d}_{13}=\matrix{c}{-3 \\ +2}, \hspace{.2in}
\mathbf{d}_{23}=\matrix{c}{-4 \\ +7}.
$$
Let us fix the starting points of the vehicles as
$$
\mathbf{x}_1[0]=\matrix{c}{+4.0 \\ -1.0},\hspace{.2in}\mathbf{x}_2[0]=\matrix{c}{+1.0 \\ -3.0},\hspace{.2in}\mathbf{x}_2[0]=\matrix{c}{-2.0 \\ +3.0}.
$$
We consider planning horizon $T=3$. As before, we use Procedure~\ref{alg:1} to calculate the sequence of sub-optimal control signals when the termination law (i.e., iteration number $S_k$) is given by Theorem~\ref{tho:2} and the sub-optimality parameter is $\alpha=0.2$. The red, blue, and green colors denotes the first, second, and third vehicle, respectively. Figure~\ref{figuretrajectory} portrays the trajectory of the vehicles in the 2-D plane. The final formation is illustrated by the triangle in black color. The codes to generate the results of this section can be found at~\cite{SimulationCodeOnline}.

We conclude this section by briefly noting that the system consisting of $N$ agents under the aforementioned cost function converges to the desired formation if and only if $\mathcal{G}^C$ is connected. This is a direct consequence of the structural rigidity of the desired formation, for more information see \cite{servatius1999constraining}. Other cost functions, such as
\begin{equation*}
\begin{split}
J_k^{(T)}((\mathbf{x}_i[k])_{i=1}^N;(\mathbf{u}_i&[k:k+T])_{i=1}^N)=\sum_{t=k}^{k+T} \bigg[  \sum\limits_{(i,j)\in\mathcal{E}^C} \big( \|\mathbf{x}_i[t]-\mathbf{x}_j[t]\|_2^2-\|\mathbf{d}_{ij}\|_2^2\big)^2+\sum_{i=1}^N v_i^2[t]\bigg],
\end{split}
\end{equation*}
can be considered as well. In this case, the system converges to the desired formation if and only if the formation is globally rigid with $N\geq 4$, see~\cite{krick2009stabilisation,anderson2010formal} and references therein.

\section{Conclusions} \label{sec:conclusion}
In this paper, we considered the dual decomposition formulation of a distributed MPC problem for systems with arbitrary dynamical couplings. More specifically, we studied the problem of calculating  performance bounds on the solution obtained from iteratively solving the dual problem in a distributed way when the iterations are terminated after $S_k$ steps at time-step $k$. Later, we commented on how the problem can be cast in an ADMM setting. We demonstrated the validity of the proposed performance bound through simulations on formation acquisition by a group of nonholonomic agents. As a future research direction, one might consider providing better performance bounds for the case where ADMM is implemented to solve the MPC problem.

\bibliographystyle{ieeetr}
\bibliography{bibfile}

\end{document}